\documentclass{amsart}
\usepackage{amssymb,mathrsfs}

\def\bB {\mathbf{B}}
\def\bC {\mathbf{C}}

\def\bN {\mathbf{N}}

\def\bR {\mathbf{R}}
\def\bS {\mathbf{S}}
\def\bT {\mathbf{T}}

\def\bZ {\mathbf{Z}}

\def\cA {\mathcal{A}}

\def\cD {\mathcal{D}}

\def\cK {\mathcal{K}}
\def\cL {\mathcal{L}}

\def\cP {\mathcal{P}}

\def\cT {\mathcal{T}}

\def\a {{\alpha}}
\def\b {{\beta}}
\def\g {{\gamma}}
\def\Ga {{\Gamma}}

\def\eps {{\epsilon}}
\def\th {{\theta}}

\def\ka {{\kappa}}
\def\l {{\lambda}}
\def\L {{\Lambda}}
\def\si {{\sigma}}

\def\om {{\omega}}
\def\Om {{\Omega}}

\def\d {{\partial}}
\def\grad {{\nabla}}
\def\Dlt {{\Delta}}

\def\rstr {{\big |}}
\def\indc {{\bf 1}}

\def\la {\langle}
\def\ra {\rangle}
\def \La {\bigg\langle}
\def \Ra {\bigg\rangle}
\def \lA {\big\langle \! \! \big\langle}
\def \rA {\big\rangle \! \! \big\rangle}
\def \LA {\bigg\langle \! \! \! \! \! \; \bigg\langle}
\def \RA {\bigg\rangle \! \! \! \! \! \; \bigg\rangle}

\def\dd{\,\mathrm{d}}
\def\rI{\mathrm{I}}

\newcommand{\Div}{\operatorname{div}}

\newcommand{\Ker}{\operatorname{Ker}}
\newcommand{\IM}{\operatorname{Im}}

\newcommand{\ba}{\begin{aligned}}
\newcommand{\ea}{\end{aligned}}

\newcommand{\be}{\begin{equation}}
\newcommand{\ee}{\end{equation}}

\newcommand{\lb}{\label}

\newtheorem{Thm}{Theorem}[section]

\newtheorem{Lem}[Thm]{Lemma}

\begin{document}


\title[Linear Boltzmann Fractional Diffusion]{Linear Boltzmann Equation\\ and Fractional Diffusion}

\author{Claude Bardos}
\address[C.B.]{Laboratoire J.-L. Lions, BP 187, 75252 Paris Cedex 05, France}
\email{claude.bardos@gmail.com}

\author{Fran\c cois Golse}
\address[F.G.]{CMLS, \'Ecole polytechnique, CNRS, Universit\'e Paris-Saclay , 91128 Palaiseau Cedex, France}
\email{francois.golse@polytechnique.edu}

\author{Ivan Moyano}
\address[I.M.]{Laboratoire J.-L. Lions, BP 187, 75252 Paris Cedex 05, France}
\email{ivan.moyano@math.cnrs.fr}

\begin{abstract}
Consider the linear Boltzmann equation of radiative transfer in a half-space, with constant scattering coefficient $\si$. Assume that, on the boundary of the half-space, the radiation intensity satisfies the Lambert (i.e. diffuse) reflection law 
with albedo coefficient $\a$. Moreover, assume that there is a temperature gradient on the boundary of the half-space, which radiates energy in the half-space according to the Stefan-Boltzmann law. In the asymptotic regime where
$\si\to+\infty$ and $1-\a\sim C/\si$, we prove that the radiation pressure exerted on the boundary of the half-space is governed by a fractional diffusion equation. This result provides an example of fractional diffusion asymptotic limit of a 
kinetic model which is based on the harmonic extension definition of $\sqrt{-\Dlt}$. This fractional diffusion limit therefore differs from most of other such limits for kinetic models reported in the literature, which are based on specific
properties of the equilibrium distributions (``heavy tails'') or of the scattering coefficient as in [U. Frisch-H. Frisch: Mon. Not. R. Astr. Not. \textbf{181} (1977), 273--280].
\end{abstract}



\maketitle


\section{Introduction}


The diffusion approximation for the linear Boltzmann equation has been known for a long time and in very different contexts, such as nuclear engineering (see chapter IX in \cite{WeinbergWigner}), or radiative transfer (see chapter III.2 in 
\cite{Pomran}). There are several proofs of the validity of the diffusion approximation, involving rather different mathematical methods: stochastic processes \cite{Hasmin, Papanico, BLP}, Hilbert expansion \cite{LK,BSS}, or chapter XXI, \S 5
in \cite{DL}, moment method \cite{BGP,BBGS} \dots

However, there are several situations where the same type of scaling limits of the linear Boltzmann equation lead to fractional (or more generally nonlocal) diffusion equation. A first class of linear Boltzmann equations leading to a fractional 
diffusion equation includes situations where the scattering rate is not uniformly large over the whole energy range of the particle system considered. See \cite{FF} for an example in radiative transfer. 

A second class of linear Boltzmann equations leading to a fractional diffusion equation includes the case of collision integrals whose equilibrium solutions have infinite second order moments: see for instance \cite{BenMelP, Mellet, MMM, Schm}.

In the present work, we give an example of a completely different type of asymptotic limit of a linear Boltzmann equation leading to a fractional diffusion equation. 

Before describing more precisely the physical problem considered in the present paper, we recall the following well-known fact from classical analysis, which the reader might find convenient to keep in mind.

\subsection{Fractional Diffusion and Harmonic Extensions}\lb{SS-HarmExt}


Let\footnote{We denote by $\bT^d$ the $d$-dimensional torus $\bT^d=\bR^d/\bZ^d$.} $f\in L^p(\bT^d)$, with $1\le p<\infty$, and let $0<\g<2$. Then
$$
(-\Dlt_x)^{\g/2}f(x)=-\d_yF(x,0)\,,
$$
where $F\equiv F(x,y)\in C_b([0,\infty);L^p(\bT^d))$ is the unique solution to the boundary value problem
$$
\left\{
\ba
{}&-\Dlt_xF(x,y)-\g^2c_\g^{\g/2}y^{2-2/\g}\d^2_yF(x,y)=0\,,&&\qquad (x,y)\in\bT^d\times(0,+\infty)\,,
\\
&F(x,0)=f(x)\,,&&\qquad x\in\bT^d\,,
\ea
\right.
$$
where
$$
c_\g=2^{-\g}|\Ga(-\g/2)|/\Ga(\g/2)
$$
(See for instance \cite{Kwasnicki}, Theorem 1.1 (j) for the analogous result, assuming that the function $f$ satisfies $f\in L^p(\bR^d)$ instead of $L^p(\bT^d)$.))

\subsection{The Model}


We seek to describe radiative transfer in a homogeneous medium filling the half-space 
$$
\{z:=(x,y)\,:\,y>0\quad\hbox{ and }x\in\bR^2\}\,.
$$
We assume that true absorption and emission are negligible effects in this medium. We assume that the only process of physical importance in the medium is scattering (of Thomson type), which means that the elementary scattering process 
at the level of particles is independent of, and does not modify the frequency of photons. The scattering coefficient $\si>0$ and scattering transition probability (also called ``phase function'' in \S 3 of \cite{Chandra}) $p\equiv p(\om,\om')$ are 
assumed to be independent of the position variable $z$. 

Under these assumptions, the radiative intensity (integrated in the frequency variable) at the position $z$ and in the direction $\om$, henceforth denoted $f\equiv f(z,\om)$, satisfies the radiative transfer equation (see chapter I in \cite{Chandra}):
\be\lb{RTEq1}
\om\cdot\grad_zf(z,\om)=\si\int_{\bS^2}p(\om,\om')(f(z,\om')-f(z,\om))\dd\om\,,\quad y>0\,,\,\,|\om|=1\,.
\ee

On the plane of equation $y=0$ (the boundary of the half-space), we assume that the radiation field obeys the Lambert reflection law, with albedo $\a$ (see formula (81) in \S 47 on p. 146 in \cite{Chandra}). On the other hand, we assume that
the temperature at the point $x$ of the plane of equation $y=0$ is $T(x)$, so that the surface emits a radiative intensity $acT(x)^4$ in all directions at the point $x$, according to the Stefan-Boltzmann law. The constant $a$ is 
$$
a:=\frac{8\pi^5k_B^4}{15c^2\hbar^3}\,,
$$
where $k_B$ is the Boltzmann constant, $c$ is the speed of light in the vacuum, and $\hbar$ is the reduced Planck constant.

In other words, the radiative intensity $f$ satisfies the boundary condition
\be\lb{BC1}
f(x,0,\om)=aT(x)^4+\frac\a\pi\int_{\bS^2}f(x,0,\om')(\om'_y)^-\dd\om'\,,\qquad\om_y>0\,.
\ee

\subsection{Scaling and Other Assumptions}


We are concerned with the following limit of the radiative transfer equation \eqref{RTEq1} with the boundary condition \eqref{BC1}:

\smallskip
\noindent
(a) $\si\gg 1$ (high scattering regime);

\noindent
(b) $0<1-\a\ll 1$ (high albedo surface).

\smallskip
We shall adopt the following mathematical setting. For any integer $d\ge 1$, set $Z:=\bT^d\times(0,+\infty)$; points in $Z$ are denoted by $z=(x,y)$ with $x\in\bT^d$ and $y>0$. Accordingly, unit tangent vectors to $Z$ are denoted 
by\footnote{We denote by $\bS^d$ the unit sphere in $\bR^{d+1}$, and by $\bB^d$ the unit ball in $\bR^d$. We also denote by $|\bS^d|$ the $d$-dimensional area of $\bS^d$, and by $|\bB^d|$ the $d$-dimensional Lebesgue 
measure of $\bB^d$. We recall that $d|\bB^d|=|\bS^{d-1}|$ for all $d\ge 1$.} $\om=(\om_x,\om_y)\in\bS^d$, with $\om_x\in\bR^d$ and $\om_y\in\bR$. One has
$$
\int_{\bS^d}\om_y^\pm\dd\om=\int_0^{\pi/2}\cos\th(\sin\th)^{d-1}\dd\th=\frac1d|\bS^{d-1}|=|\bB^d|\,.
$$
We shall henceforth use the following notation:
$$
\la\phi\ra:=\frac1{|\bS^d|}\int_{\bS^d}\phi(\om)\dd\om\,,\qquad\phi\in C(\bS^d)\,,
$$
and
$$
\lA\psi\rA_\pm:=\frac1{|\bB^d|}\int_{\bS^d}\psi(\om)\om_y^\pm\dd\om\,,\qquad\phi\in C(\bS^d)\,.
$$

\smallskip
We assume that condition (b) is realized by taking the albedo coefficient $\a$ of $\d Z$ of the form
$$
\a:=\frac{\ka\si}{1+\ka\si}
$$
where $\ka>0$ is a constant, and we set
$$
S(x):=(1+\ka\si)aT(x)^4\,.
$$
Since $\b_2=\pi$, the boundary condition \eqref{BC1} is recast as
\be\lb{BC}
f(x,0,\om)=\frac{S(x)}{1+\ka\si}+\frac{\ka\si}{1+\ka\si}\lA f(x,0,\cdot)\rA_-\,,\quad\om_y>0\,.
\ee
Observing that the right hand side of \eqref{BC} is independent of $\om$, this boundary condition implies that $f(x,0,\om)$ is independent of $\om$ for $\om_y>0$, so that
$$
f(x,0,\om)=\frac{f(x,0,\om)}{1+\ka\si}+\frac{\ka\si}{1+\ka\si}\lA f(x,0,\cdot)\rA_+\,.
$$
Substituting this expression on the left hand side of \eqref{BC}, we arrive at the following equivalent formulation of the boundary condition \eqref{BC}:
\be\lb{BCC}
f(x,0,\om)=S(x)-\frac{\ka\si}{|\bB^d|}\int_{\bS^d}f(x,0,\om')\om'_y\dd\om'\,,\quad\om_y>0\,.
\ee

We further assume that $p\equiv p(\om,\om')$ is a measurable function defined a.e. on $\bS^d\times\bS^d$ satisfying the following condition:
\be\lb{H1p}
p(\om,\om')=p(\om',\om)>0\hbox{ a.e. on }\bS^d\times\bS^d\,,\quad\int_{\bS^d}p(\om,\om')\dd\om'=1\hbox{ for a.e. }\om\in\bS^d\,.
\ee
We shall denote by $\cL$ the bounded linear operator defined on $L^\infty(\bS^d)$ by the formula
$$
\cL\phi(\om)=\int_{\bS^d}p(\om,\om')(\phi(\om)-\phi(\om'))\dd\om'=\phi(\om)-\int_{\bS^d}p(\om,\om')\phi(\om')\dd\om'\,.
$$

With these elements of notation, the boundary value problem for the radiative transfer equation \eqref{RTEq1} with boundary condition \eqref{BC1} takes the form
\be\lb{RTPb}
\left\{
\ba
{}&\om\cdot\grad_zf_\si(z,\om)+\si(\cL f_\si)(z,\cdot)=0\,,&&\quad(z,\om)\in Z\times\bS^d\,,
\\	\\
&f_\si(x,0,\om)=\frac{S(x)}{1+\ka\si}+\frac{\ka\si}{1+\ka\si}\lA f_\si(x,0,\cdot)\rA_-\,,&&\quad x\in\bT^d\,,\,\,\om_y>0\,,
\ea
\right.
\ee
where $$
(\cL f_\si)(z,\om):=(\cL f_\si(z,\cdot))(\om)\,.
$$

\medskip
The purpose of the present work is to study the boundary value problem \eqref{RTPb} in the asymptotic regime $\si\to+\infty$.


\section{Main Result}


\subsection{Heuristic approach}

 
In order to gain some intuition on this problem, we apply the Hilbert expansion method. This method is named after Hilbert on the basis of \cite{Hilbert1912}, where it has been used for the first time on the Boltzmann equation in the context 
of the kinetic theory of gases. Its application to the linear Boltzmann equation is discussed in detail in \cite{LK, BSS}. 

Thus, we seek $f_\si$ as the formal power series
\be\lb{Hilbert}
f_\si(z,\om)=\sum_{j\ge 0}\si^{-j}f_j(z,\om)\,,\quad \grad^m_zf_j\in C(Z\times\bS^d)\hbox{ for all integers }j,m\ge 0\,.
\ee
Substituting this ansatz in the radiative transfer equation leads to the sequence of integral equations:

\smallskip
\noindent
\textit{Order $O(\si)$:}
$$
\cL f_0(z,\om)=0\,,
$$
\textit{Order $O(1)$:}
$$
\cL f_1(z,\om)=-\om\cdot\grad_zf_0\,,
$$
\bigskip
\dots\,\dots\,\dots\,\dots\,\dots\,\dots\,\dots\,\dots\,\dots\,\dots\,\dots\,\dots\,\dots\,\dots\,\dots\,\dots\,

\medskip
\noindent
\textit{Order $O(\si^{-j})$:}
$$
\cL f_{j+1}(z,\om)=-\om\cdot\grad_zf_j\,.
$$
\bigskip
\dots\,\dots\,\dots\,\dots\,\dots\,\dots\,\dots\,\dots\,\dots\,\dots\,\dots\,\dots\,\dots\,\dots\,\dots\,\dots\,

\bigskip
\begin{Lem}\lb{L-KerL}
Under assumption \eqref{H1p}, the linear operator $\cL$ is bounded on $L^2(\bS^d;\bC^d)$ and satisfies
$$
\|\cL\|\le 2\,,\quad\hbox{ and }\quad\Ker\cL=\{\hbox{ constants }\}=\bC\,.
$$
\end{Lem}

\begin{proof}
Because of the last equality in \eqref{H1p}, one has $\bR\subset\Ker\cL$.

For each $\phi\in L^2(\bS^d)$, one has
\be\lb{phiLpsi}
\ba
\int_{\bS^d}\phi(\om)(\cL\psi)(\om)\dd\om=&\iint_{\bS^d\times\bS^d}\phi(\om)p(\om,\om')(\psi(\om)-\psi(\om'))\dd\om\dd\om'
\\
=&\iint_{\bS^d\times\bS^d}\phi(\om')p(\om,\om')(\psi(\om')-\psi(\om))\dd\om\dd\om'
\\
=&\iint_{\bS^d\times\bS^d}\tfrac12p(\om,\om')(\phi(\om)-\phi(\om'))(\psi(\om)-\psi(\om'))\dd\om\dd\om'\,.
\ea
\ee
The second equality above follows from the symmetry of $p$ in \eqref{H1p}, while the third equality is obtained by adding the right hand sides of the first and the second equalities.

If $\phi\in\Ker\cL$, then 
$$
\iint_{\bS^d\times\bS^d}p(\om,\om')(\phi(\om)-\phi(\om'))^2\dd\om\dd\om'=0
$$
because of \eqref{phiLpsi} with $\phi=\psi$. Since $p>0$ a.e. on $\bS^d\times\bS^d$,
$$
\phi(\om)-\phi(\om')=0\hbox{ for a.e. }(\om,\om')\in\bS^d\times\bS^d\,.
$$
Averaging both sides of this equality in $\om'$ implies that
$$
\phi(\om)=\la\phi\ra\hbox{ for a.e. }\om\in\bS^d\,.
$$
Thus $\phi$ is a.e. a constant, and this proves that $\Ker\cL=\{\hbox{ constants }\}$.

Finally
$$
\ba
\int_{\bS^d}|(\cL\phi)(\om)|^2\dd\om=\int_{\bS^d}\left|\int_{\bS^d}p(\om,\om')(\phi(\om)-\phi(\om'))\dd\om'\right|^2\dd\om
\\
\le\int_{\bS^d}\left(\int_{\bS^d}p(\om,\om')\dd\om'\right)\left(\int_{\bS^d}p(\om,\om')|\phi(\om)-\phi(\om')|^2\dd\om'\right)\dd\om
\\
=\iint_{\bS^d}p(\om,\om')|\phi(\om)-\phi(\om')|^2\dd\om\dd\om'=2\int_{\bS^d}\overline{\phi(\om)}(\cL\phi)(\om)\dd\om&\,,
\ea
$$
where the inequality follows from the Cauchy-Schwarz inequality applied to the inner integral, and the last equality from \eqref{phiLpsi} with $\overline{\phi}=\psi$. Applying the Cauchy-Schwarz inequality to the last right hand side, 
one obtains
$$
\int_{\bS^d}|(\cL\phi)(\om)|^2\dd\om\le 2\|\phi\|_{L^2(\bS^d)}\|\cL\phi\|_{L^2(\bS^d)}\,,
$$
which implies that $\|\cL\|\le 2$.
\end{proof}

\medskip
We conclude from this lemma the solution to the equation at order $O(\si)$:

\smallskip
\noindent
\textit{Order $O(\si)$:}
$$
f_0(z,\om)=\rho_0(z)\,,\quad\hbox{ for a.e. }(z,\om)\in Z\times\bS^d\,.
$$

\bigskip
Next we study the equation at order $O(1)$ for $f_1$. We shall need the following additional assumption on $p$:
\be\lb{H2p}
p\in L^2(\bS^d\times\bS^d)\,.
\ee 
Denoting by $\cK$ the linear operator defined by
\be\lb{DefK}
(\cK\phi)(\om):=\int_{\bS^d}p(\om,\om')\phi(\om')\dd\om'\,,\qquad\om\in\bS^d\,,
\ee
one has
$$
\cL=\rI-\cK\,.
$$
\begin{Lem}\lb{L-ImL}
The operator $\cL$ is bounded and of Fredholm type on $L^2(\bS^d;\bC^d)$, and satisfies
$$
\cL=\cL^*\,,\qquad\IM\cL=\bC^\perp=\left\{\psi\in L^2(\bS^d;\bC^d)\hbox{ s.t. }\int_{\bS^d}\psi(\om)\dd\om=0\right\}\,.
$$
\end{Lem}

\begin{proof}
The operator $\cK$ is a self-adjoint Hilbert-Schmidt operator on $L^2(\bS^d;\bC^d)$ since its integral kernel $p$ is a symmetric real-valued square-integrable function by \eqref{H1p}-\eqref{H2p}. In particular $\cK$ is a compact operator 
on $L^2(\bS^d;\bC^d)$ (see section 6.4.2 and Theorem 6.12 in \cite{Brezis}). This implies that $\cL$ is a self-adjoint, bounded operator on $L^2(\bS^d;\bC^d)$ of Fredholm type (see section 6.4.1 in \cite{Brezis}). 

Fredholm's alternative (Theorem 6.6 in \cite{Brezis}) states that 
$$
\IM\cL=(\Ker\cL)^\perp
$$
and we conclude the proof with the characterization of $\Ker\cL$ in Lemma \ref{L-KerL}.
\end{proof}

\smallskip
Since $\la\om\ra=0$, Fredholm's alternative implies that there exists a unique function $\Om\equiv\Om(\om)\in L^2(\bS^d)$ such that
\be\lb{DefOm}
\la\Om\ra=0\quad\hbox{ and }\quad(\cL\Om)(\om)=\om\,,\quad\hbox{ for a.e. }\om\in\bS^d\,.
\ee

\smallskip
Combining Lemmas \ref{L-KerL} and \ref{L-ImL} leads to the following result for the solution of the equation at order $O(1)$:

\smallskip
\noindent
\textit{Order $O(1)$:}
$$
f_1(z,\om)=\rho_1(z)-\Om(\om)\cdot\grad\rho_0(z)\,,\quad\hbox{ for a.e. }(z,\om)\in Z\times\bS^d\,.
$$

\smallskip
Higher order terms in the expansion are found in the same way, by solving successively for $f_{j+1}$ the integral equations
$$
\cL f_{j+1}=-\om\cdot\grad_zf_j
$$
for $j\ge 1$. We shall not pursue this line of investigation, since only $f_0$ and $f_1$ will be used in the present section.

\smallskip
At this point, we introduce a last assumption on the scattering transition probability $p$, specifically, we require that $p$ is rotationally invariant: for each $Q$ in the orthogonal group $O_{d+1}(\bR)$,
\be\lb{H3p}
p(Q\om,Q\om')=p(\om,\om')\,,\quad\hbox{ for a.e. }(\om,\om')\in\bS^d\times\bS^d\,.
\ee

Under this assumption, it has been proved in \cite{BSS} that the diffusion matrix\footnote{See formulas (40)-(44) on p. 624 in \cite{BSS}. Actually, one can deduce from equation (42) in \cite{BSS} that $\Om$ is of the form $\Om(\om)=\l\om$:
see Lemma 3 in \cite{DesvFG}, or Lemma 8 in Appendix 1 of \cite{FGBraga}. With $\Om$ of this form, one immediately concludes that $\la\om\otimes\Om\ra=\frac{\l}{d+1}\rI$.}
$$
\la\om\otimes\Om\ra\hbox{ is of the form }\la\om\otimes\Om\ra=\frac{\la\om\cdot\Om\ra}{d+1}\rI\,.
$$
Besides
$$
\la\om\cdot\Om\ra=\la(\cL\Om)\cdot\Om\ra=\tfrac12\iint_{\bS^d\times\bS^d}p(\om,\om')|\Om(\om)-\Om(\om')|^2\dd\om\dd\om'>0\,.
$$
Notice that this inequality is strict: otherwise, one would have 
$$
p(\om,\om')(\Om(\om)-\Om(\om'))=0\quad\hbox{ for a.e. }(\om,\om')\in\bS^d\times\bS^d\,.
$$
This would imply that
$$
0=\int_{\bS^d}p(\om,\om')(\Om(\om)-\Om(\om'))\dd\om'=(\cL\Om)(\om)=\om
$$
which is obviously a contradiction.

From the equation at order $O(\si)$, i.e.
$$
\cL f_2=-\om\cdot\grad_zf_1\,,
$$
we conclude that
$$
\Div_z\la\om f_1\ra=0\,.
$$
Otherwise, one would have $\om\cdot\grad_zf_1\notin\IM\cL$ according to Lemma \ref{L-ImL}. With the expression for $f_1$ obtained above, this equality takes the form
$$
\Div_z\left(-\frac{\la\om\cdot\Om\ra}{d+1}\grad_z\rho_0\right)=0\,,
$$
or equivalently
\be\lb{EqRho0}
-\Dlt\rho_0(z)=0\,,\quad z\in Z\,.
\ee

Moreover, if the Hilbert expansion \eqref{Hilbert} holds at the boundary, then we deduce from the equivalent formulation \eqref{BCC} of the boundary condition in \eqref{RTPb} that
\be\lb{BCRho0}
\ba
\rho_0(x,0)=&S(x)-\frac{\ka\si}{|\bB^d|}\int_{\bS^d}\left(\rho_1(x,0)-\frac1\si\Om(\om')\cdot\grad\rho_0(x,0)\right)\om'_y\dd\om'
\\
=&S(x)+\frac{\ka|\bS^{d}|}{|\bB^d|}\la\om_y\Om\ra\cdot\grad\rho_0(x,0)
\\
=&S(x)+\frac{\ka|\bS^d|\la\om\cdot\Om\ra}{|\bB^d|(d+1)}\d_y\rho_0(x,0)\,,\qquad x\in\bT^d\,.
\ea
\ee

Returning to the original variables, and to the case of physical interest where $d=2$, we recall that the radiation pressure $\cP(x)$ at the point $x$ of $\d Z$ is defined by the identity
$$
\frac1c\int_{\bS^2}\om\otimes\om f(x,0,\om)\dd\om=\cP(x)\rI\,.
$$
(See for instance formula (1.19) in \cite{Pomran}). Therefore, to leading order in $1/\si$, one has
$$
\cP(x)\simeq\frac{4\pi}{3c}\rho_0(x,0)\,.
$$

Since $\frac{4\pi}{3c}\rho_0$ is an harmonic extension of $\cP$, applying the formula for the fractional diffusion operator recalled in section \ref{SS-HarmExt} leads to the following equation for the radiation pressure field on the boundary:
\be\lb{EqRadP}
\tfrac34(1-\a)\cP(x)+\a\la\om\cdot\Om\ra\frac1{\si}(-\Dlt_x)^{1/2}\cP(x)=\frac{a\pi T(x)^4}{c}\,,\qquad x\in\bT^2\,.
\ee

Assumption \eqref{H3p} is verified by most of the scattering transition probabilities $p$ used in practice, such as 
$$
\ba
p(\om,\om')&=\frac1{|\bS^d|}\,,\qquad&&\hbox{(isotropic scattering)}
\\
p(\om,\om')&=\tfrac{3}{16\pi}(1+(\om\cdot\om')^2)\,,\qquad&&\hbox{(Rayleigh phase function for $d=2$)}
\ea
$$
(see formulas (30)-(31)-(33) in \S 3 or formula (192) in \S 16 of \cite{Chandra}). More generally, all scattering transition probabilities of the form
\be\lb{FormP}
p(\om,\om')=P((\om\cdot\om')^2)\,,
\ee
under the normalizing condition
$$
\int_0^1P(\mu^2)(1-\mu^2)^{\frac{d}2-1}\dd\mu=\frac1{2|\bS^{d-1}|}\,,
$$
satisfy \eqref{H3p}. In that case
$$
\int_{\bS^d}P((\om\cdot\om')^2)\om'\dd\om'=0
$$
(since the integral the right hand side of this last equality must be invariant under the substitution $\om'\mapsto-\om'$), so that
$$
\cL\om=\om\,,\quad\hbox{ and therefore }\quad\Om(\om)=\om\,.
$$
In that case
$$
\la\om\cdot\Om\ra=1\,,
$$
and the diffusion coefficient in the usual diffusion approximation of the linear Boltzmann equation is
$$
\frac{\la\om\cdot\Om\ra}{(d+1)\si}=\frac1{(d+1)\si}\,.
$$

In particular, the equation \eqref{EqRadP} satisfied by the radiation pressure field on the boundary $\d Z$ is, in this case,
$$
\tfrac34(1-\a)\cP(x)+\frac\a{\si}(-\Dlt_x)^{1/2}\cP(x)=\frac{a\pi T(x)^4}{c}\,,\qquad x\in\bT^2\,.
$$

\subsection{The Limit Theorem}


The analysis based on Hilbert's expansion presented in the previous section is only formal. A rigorous analysis of the problem based on the moment method for kinetic models leads to the following result.

\begin{Thm}\lb{T-Ppal}
Assume that $p\equiv p(\om,\om')$ is a measurable function defined a.e. on $\bS^d\times\bS^d$, satisfying \eqref{H1p}-\eqref{H2p}-\eqref{H3p}. Denote by $\Om$ the unique element of $L^2(\bS^d)$ defined by \eqref{DefOm}. 
Let $S\in W^{1,\infty}(\bT^d)$. 

\smallskip
\noindent
(a) For each $\si>0$, the boundary value problem \eqref{RTPb} has a unique solution
$$
f_\si\in L^\infty(\bT^d\times(0,+\infty)\times\bS^d)\,;
$$
this solution satisfies
$$
\|f_\si\|_{L^\infty(\bT^d\times(0,+\infty)\times\bS^d)}\le\|S\|_{L^\infty(\bT^d)}\,,
$$
and
$$
\|\d_{x_j}f_\si\|_{L^\infty(\bT^d\times(0,+\infty)\times\bS^d)}\le\|\d_{x_j}S\|_{L^\infty(\bT^d)}
$$
for all $j=1,\ldots,d$.

\noindent
(b) In the limit as $\si\to+\infty$, one has
$$
f_\si\to\rho\equiv\rho(z)\hbox{ in }L^\infty(\bT^d\times(0,+\infty)\times\bS^d)\hbox{ weak}-*
$$
and
$$
\si(f_\si-\la f_\si\ra)\to-\Om\cdot\grad\rho\hbox{ in }L^2(\bT^d\times(0,+\infty)\times\bS^d)\hbox{ weak}\,,
$$
where $\rho\equiv\rho(z)$ satisfies the properties
$$
\rho\in L^\infty(\bT^d\times(0,+\infty))\,,\qquad\grad\rho\in L^2(\bT^d\times(0,+\infty))\,,
$$
and is the weak solution of the boundary value problem
\be\lb{Diff}
\left\{
\ba
{}&-\Dlt_z\rho(z)=0\,,\qquad &&z=(x,y)\in\bT^d\times(0,+\infty)\,,
\\
&\left(\rho-\ka\la\om\cdot\Om\ra\tfrac{|\bB^{d+1}|}{|\bB^d|}\d_y\rho\right)\rstr_{y=0}=S\,,&&x\in\bT^d\,.
\ea
\right.
\ee
(c) In the limit as $\si\to+\infty$, one has
$$
f_\si\rstr_{y=0}\to R\quad\hbox{ in }L^2(\bT^d\times\bS^d;|\om_y|\dd x\dd\om)\hbox{ weak,}
$$
and
$$
\si\la\om_yf_\si\ra\rstr_{y=0}\to-\tfrac{\la\om\cdot\Om\ra}{d+1}\d_y\rho\rstr_{y=0}\hbox{ in }H^{-1/2}(\bT^d)\hbox{ weak,}
$$
where $R\equiv R(x)$ is the solution to the fractional diffusion equation
\be\lb{FracDiff}
R(x)+\ka\la\om\cdot\Om\ra\tfrac{|\bB^{d+1}|}{|\bB^d|}(-\Dlt_x)^{1/2}R(x)=S(x)\,,\qquad x\in\bT^d\,.
\ee
\end{Thm}


\section{Proof of Theorem \ref{T-Ppal}}


\subsection{Step 1}


Consider the boundary value problem
$$
\left\{
\ba
{}&\l h+\om\cdot\grad_zh+\si h=\si Q(z,\om)\,,
\\
&h(x,0,\om)=H(x)\,,\quad\om_y>0\,,
\ea
\right.
$$
where $Q\in L^\infty(Z\times\bS^d)$ and $H\in L^\infty(\bT^d)$. 

Its unique solution is given by the method of characteristics
$$
\ba
h(z,\om)=&H(x-\tfrac{y}{\om_y}\om_x)e^{-(\l+\si)y/\om_y}&&
\\
&+\int_0^{y/\om_y}\si e^{-(\l+\si)t}Q(z-t\om,\om)\dd t\,,&&\qquad\om_y>0\,,
\\
h(z,\om)=&\int_0^\infty\si e^{-(\l+\si)t}Q(z-t\om,\om)\dd t\,,&&\qquad\om_y<0\,.
\ea
$$
Setting
$$
\th:=e^{-(\l+\si)y/\om_y}
$$
we find that, for a.e. $(z,\om)\in Z\times\bS^d$,
$$
\ba
|h(z,\om)|\le\max(\th\|H\|_{L^\infty}+\tfrac{\si}{\si+\l}(1-\th)\|Q\|_{L^\infty},\tfrac{\si}{\si+\l}\|Q\|_{L^\infty})
\\
=\max(\tfrac{\l\th}{\si+\l}\|H\|_{L^\infty}+\tfrac{\si}{\si+\l}((1-\th)\|Q\|_{L^\infty}+\th\|H\|_{L^\infty}),\tfrac{\si}{\si+\l}\|Q\|_{L^\infty})
\\
\le\max(\tfrac{\l}{\si+\l}\|H\|_{L^\infty}+\tfrac{\si}{\si+\l}\max(\|Q\|_{L^\infty},\|H\|_{L^\infty}),\tfrac{\si}{\si+\l}\|Q\|_{L^\infty})&\,.
\ea
$$

This inequality implies, on the one hand, that, for each $H\in L^\infty(\bT^d)$, the map
$$
Q\mapsto\cK h
$$
is a contraction in $L^\infty(Z\times\bS^d)$ with Lipschitz constant $\le\tfrac{\si}{\l+\si}$. Indeed, denoting by $h_1$ and $h_2$ corresponding to $Q_1$ and $Q_2$ respectively, one has
$$
|(h_1-h_2)(z,\om)|\le\tfrac{\si}{\si+\l}\|Q_1-Q_2\|_{L^\infty}\,,
$$
as a consequence of the previous inequality. Indeed, by linearity $h_1-h_2$ is a solution of the boundary value problem above with source term $Q_1-Q_2$ and boundary data $H=0$. 

Using assumption \eqref{H1p}, we see that the operator $\cK$ defined in \eqref{DefK} is a bounded operator on $L^\infty(\bS^d)$, satisfying
$$
\|\cK\|_{\cL(L^\infty(\bS^d))}\le 1\,.
$$
Hence
$$
\|\cK h_1-\cK h_2\|_{L^\infty}\le\|h_1-h_2\|_{L^\infty}\le\tfrac{\si}{\si+\l}\|Q_1-Q_2\|_{L^\infty}\,.
$$
Therefore, the map $Q\mapsto\cK h$ has a unique fixed point in $L^\infty(Z\times\bS^d)$. 

In other words, there exists a unique solution $g\in L^\infty(Z\times\bS^d)$ to the boundary value problem
\be\lb{Pblg}
\left\{
\ba
{}&\l g+\om\cdot\grad_zg+\si\cL g=0\,,&&\quad(z,\om)\in Z\times\bS^d\,,
\\
&g(x,0,\om)=H(x)\,,&&\quad\om_y>0\,,
\ea
\right.
\ee
and this solution satisfies the bound
$$
\|g\|_{L^\infty}\le\max(\tfrac{\l}{\si+\l}\|H\|_{L^\infty}+\tfrac{\si}{\si+\l}\max(\|\cK g\|_{L^\infty},\|H\|_{L^\infty}),\tfrac{\si}{\si+\l}\|\cK g\|_{L^\infty})\,.
$$

It is obviously impossible that
$$
\|\cK g\|_{L^\infty}>\|H\|_{L^\infty}\,.
$$
Indeed, since
$$
\|\cK g\|_{L^\infty}\le\|g\|_{L^\infty}\,,
$$
the inequality above would imply that
$$
\|\cK g\|_{L^\infty}\le\tfrac{\l}{\si+\l}\|H\|_{L^\infty}+\tfrac{\si}{\si+\l}\|\cK g\|_{L^\infty}\,,
$$
which would imply in turn
$$
\|\cK g\|_{L^\infty}\le\|H\|_{L^\infty}\,,
$$
in contradiction with our assumption. Hence
$$
\|\cK g\|_{L^\infty}\le\|H\|_{L^\infty}\,,
$$
and therefore
$$
\ba
\|g\|_{L^\infty}\le\max(\tfrac{\l}{\si+\l}\|H\|_{L^\infty}+\tfrac{\si}{\si+\l}\max(\|\cK g\|_{L^\infty},\|H\|_{L^\infty}),\tfrac{\si}{\si+\l}\|\cK g\|_{L^\infty})
\\
=\max(\tfrac{\l}{\si+\l}\|H\|_{L^\infty}+\tfrac{\si}{\si+\l}\|H\|_{L^\infty},\tfrac{\si}{\si+\l}\|\cK g\|_{L^\infty})=\|H\|_{L^\infty}&\,.
\ea
$$


\subsection{Step 2}


With step 1, for each $\l>0$ and $\si>0$, we have constructed a linear map
$$
\cT_{\l,\si}:\, L^\infty(\bT^d)\mapsto L^\infty(Z\times\bS^d)
$$
defined by the formula
$$
\cT_{\l,\si}H=g\,,
$$
where $g$ is the solution to the boundary value problem \eqref{Pblg}. We have also proved that
$$
\|\cT_{\l,\si}H\|_{L^\infty}\le\|H\|_{L^\infty}\,.
$$

Now we seek to solve the boundary value problem
\be\lb{Pblf}
\left\{
\ba
{}&\l f+\om\cdot\grad_zf+\si\cL f=0\,,&&\quad(z,\om)\in Z\times\bS^d\,,
\\
&f(x,0,\om)=\frac1{1+\b}S(x)+\frac{\b}{1+\b}\lA f(x,0,\cdot)\rA_-\,,&&\quad\om_y>0\,,
\ea
\right.
\ee

Consider the map
$$
\cA_{\l,\si,\b}:\,L^\infty(\bT^d)\to L^\infty(\bT^d)
$$
defined by the formula
$$
\cA_{\l,\si,\b}F(x):=\si\LA\int_0^\infty e^{-(\l+\si)t}\cK\left(\cT_{\l,\si}\left(\frac{S}{1+\b}+\frac{\b F}{1+\b}\right)\right)(x-t\om,\om)\dd t\RA_-\,.
$$
In terms of the operator $\cK$ defined in \eqref{DefK}, one has obviously
$$
(\cA_{\l,\si,\b}F_1-\cA_{\l,\si,\b}F_2)(x)=\frac{\b\si}{1+\b}\LA\int_0^\infty e^{-(\l+\si)t}\cK(\cT_{\l,\si}(F_1-F_2))(x-t\om,\om)\dd t\RA_-\,,
$$
so that, for a.e. $x\in\bT^d$,
$$
\ba
|(\cA_{\l,\si,\b}F_1-\cA_{\l,\si,\b}F_2)(x)|\le\frac{\b\si}{1+\b}\|\cT_{\l,\si}(F_1-F_2)\|_{L^\infty}\int_0^\infty e^{-(\l+\si)t}\dd t
\\
\le\frac{\b\si}{1+\b}\|F_1-F_2\|_{L^\infty}\int_0^\infty e^{-(\l+\si)t}\dd t
\\
=\frac{\b}{1+\b}\frac{\si}{\l+\si}\|F_1-F_2\|_{L^\infty}&\,.
\ea
$$
In other words, $\cA_{\l,\si,\b}$ is a contraction in the Banach space $L^\infty(\bT^d)$ with Lipschitz constant $\le\frac{\b}{1+\b}\frac{\si}{\l+\si}<1$.

By the fixed point theorem, there exists a unique $F\in L^\infty(\bT^d)$ such that
$$
\cA_{\l,\si,\b}F=F\,.
$$
The solution to the boundary value problem \eqref{Pblf} is given by the formula
$$
f:=\cT_{\l,\si}\left(\frac{S}{1+\b}+\frac{\b F}{1+\b}\right)\,.
$$

The fixed point $F$ satisfies in particular the bound
$$
\|F-\cA_{\l,\si,\b}0\|_{L^\infty}\le\frac{\b}{1+\b}\frac{\si}{\l+\si}\|F\|_{L^\infty}\,.
$$
Since
$$
|\cA_{\l,\si,\b}0(x)|=\si\LA\int_0^\infty e^{-(\l+\si)t}\cK\left(\cT_{\l,\si}\frac{S}{1+\b}\right)(x-t\om,\om)\dd t\RA_-\,,
$$
one has
$$
\|\cA_{\l,\si,\b}0\|_{L^\infty}\le\frac{\si}{1+\b}\|S\|_{L^\infty}\int_0^\infty e^{-(\l+\si)t}\dd t=\frac1{1+\b}\frac{\si}{\l+\si}\|S\|_{L^\infty}\,,
$$
so that
$$
\|F\|_{L^\infty}\le\frac1{1+\b}\frac{\si}{\l+\si}\|S\|_{L^\infty}+\frac{\b}{1+\b}\frac{\si}{\l+\si}\|F\|_{L^\infty}\,.
$$

Therefore
$$
\|F\|_{L^\infty}\le\frac{\si}{\l(1+\b)+\si}\|S\|_{L^\infty}\,,
$$
which implies in turn
$$
\ba
\|f\|_{L^\infty}&=\left\|\cT_{\l,\si}\left(\frac{S}{1+\b}+\frac{\b F}{1+\b}\right)\right\|_{L^\infty}
\\
&\le\left(\frac1{1+\b}+\frac{\b}{1+\b}\frac{\si}{\l(1+\b)+\si}\right)\|S\|_{L^\infty}\le\|S\|_{L^\infty}\,.
\ea
$$


\subsection{Step 3}


Call $f_\l$ the unique solution of the boundary value problem \eqref{Pblf} obtained in step 2, for which we have obtained the bound
\be\lb{Boundfl}
\|f_\l\|_{L^\infty}\le\|S\|_{L^\infty}\,,\qquad\l>0\,,
\ee
which is uniform in $\l>0$.

Now we consider the boundary value problem
\be\lb{Pbf}
\left\{
\ba
{}&\om\cdot\grad_zf+\si\cL f=0\,,&&\quad(z,\om)\in Z\times\bS^d\,,
\\
&f(x,0,\om)=\frac1{1+\b}S(x)+\frac{\b}{1+\b}\lA f(x,0,\cdot)\rA_-\,,&&\quad\om_y>0\,,
\ea
\right.
\ee
of which we seek a solution by passing to the limit in $f_\l$ as $\l\to 0$.

On account of the uniform bound \eqref{Boundfl}, the Banach-Alaoglu theorem implies the existence of a sequence $\l_n\to 0$ such that
$$
f_{\l_n}\to f\quad\hbox{ in }L^\infty(Z\times\bS^d)\hbox{ weak}-*\,.
$$
One has obviously
$$
\l_n f_{\l_n}+\om\cdot\grad_zf_{\l_n}+\si\cL f_{\l_n}\to\om\cdot\grad_zf+\si\cL f
$$
in $\cD'(\bT^d\times(0,+\infty)\times\bS^d)$, so that
$$
\om\cdot\grad_zf+\si\cL f=0\quad\hbox{ dans }\cD'(\bT^d\times(0,+\infty)\times\bS^d)\,.
$$

For each $\phi\in W^{1,1}(\bT^d)$, one has
$$
\ba
\d_y\left(\om_y\int_{\bT^d}f_{\l_n}(x,y,\om)\phi(x)\dd x\right)=&-\l_n\int_{\bT^d}f_{\l_n}(x,y,\om)\dd x
\\
&+\int_{\bT^d}\om_x\cdot\grad\phi(x)f_{\l_n}(x,y,\om)\dd x
\\
&-\si\int_{\bT^d}\cL f_{\l_n}(x,y,\om)\phi(x)\dd x\,.
\ea
$$
Therefore
$$
\int_{\bT^d}f_{\l_n}(x,y,\om)\phi(x)\dd x\quad\hbox{ et }\quad\d_y\left(\om_y\int_{\bT^d}f_{\l_n}(x,y,\om)\phi(x)\dd x\right)
$$
are bounded in $L^\infty((0,+\infty)\times\bS^d)$. By Ascoli-Arzel\`a's theorem,
$$
\om_yf_{\l_n}(x,0,\om)\to\om_yf(x,0,\om)\quad\hbox{ in }L^\infty(\bS^d;W^{-1,\infty}(\bT^d))\hbox{ weak}-*\,,
$$
and
$$
\lA f_{\l_n}\rA_-(\cdot,0)\to\lA f\rA_-(\cdot,0)\quad\hbox{ dans }W^{-1,\infty}(\bT^d)\hbox{ weak}-*\,.
$$

By passing to the limit in the boundary condition 
$$
f_{\l_n}(x,0,\om)=\frac1{1+\b}S(x)+\frac{\b}{1+\b}\lA f_{\l_n}(x,0,\cdot)\rA_-\,,\quad\om_y>0\,,\,\,\,x\in\bT^d\,,
$$
we conclude that
$$
\om_y^+f\rstr_{y=0}=\frac{\om_y^+S}{1+\b}+\frac{\b\om_y^+}{1+\b}\lA f\rA_-\rstr_{y=0}\quad\hbox{ in }L^\infty(\bS^d;W^{-1,\infty}(\bT^d))\,.
$$
In particular
$$
f(\cdot,0,\om)=\frac{S}{1+\b}+\frac{\b}{1+\b}\lA f\rA_-(\cdot,0)\quad\hbox{ in }W^{-1,\infty}(\bT^d))
$$
for a.e. $\om\in\bS^d$ such that $\om_y>0$.

The uniform bound \eqref{Boundfl} obviously implies that the solution to the boundary value problem \eqref{Pbf} satisfies the bound
$$
\|f\|_{L^\infty}\le\|S\|_{L^\infty}\,.
$$


\subsection{Step 4}


In this step, we check that, for each $S\in L^\infty(\bT^d)$, there exists at most one weak solution $f$ of the boundary value problem
$$
\left\{
\ba
{}&\om\cdot\grad_zf+\si\cL f=0\,,&&\quad(z,\om)\in Z\times\bS^d\,,
\\
&f(x,0,\om)=\frac1{1+\b}S(x)+\frac{\b}{1+\b}\lA f(x,0,\cdot)\rA_-\,,&&\quad\om_y>0\,,
\ea
\right.
$$
in the space $L^\infty(\bT^d\times(0,+\infty)\times\bS^d)$.

By linearity, it is enough to show that, if $g\in L^\infty(\bT^d\times(0,+\infty)\times\bS^d)$ is a solution to the boundary value problem
$$
\left\{
\ba
{}&\om\cdot\grad_zg+\si\cL g=0\,,&&\quad(z,\om)\in Z\times\bS^d\,,
\\
&g(x,0,\om)=\frac{\b}{1+\b}\lA g(x,0,\cdot)\rA_-\,,&&\quad\om_y>0\,,
\ea
\right.
$$
then $g=0$, a.e.. 

Denote the sequence of Fourier coefficients of $g$ in the $y$ variable by
$$
\hat g(k,y,\om)=\int_{\bT^d}e^{-2\pi ik\cdot x}g(x,y,\om)\dd x\,,\quad k\in\bZ^d\,.
$$
Then, for each $k\in\bZ^d$, on a
$$
\left\{
\ba
{}&\om_y\d_y\hat g(k,y,\om)+i2\pi k\cdot\om_x\hat g(k,y,\om)+\si\cL g(k,y,\om)=0\,,&&\quad y>0\,,\,\,|\om|=1\,,
\\
&\hat g(k,0,\om)=\frac{\b}{1+\b}\lA g\rA_-(k,0)\,,\quad\om_y>0\,.
\ea
\right.
$$
For each $k$ and for a.e. $\om\in\bS^d$, the function $y\mapsto g(k,y,\om)$ belongs therefore to $W^{1,\infty}(0,+\infty)$.

Multiplying both sides of the differential equation above by $\overline{\hat g(k,y,\om)}$ leads to the identity
$$
(\overline{\hat g}\om_y\d_y\hat g)(k,y,\om)+i2\pi k\cdot\om_x|\hat g(k,y,\om)|^2+\si\overline{\hat g(k,y,\om)}\cL\hat g(k,y,\om)=0\,.
$$
Taking the real part of the right hand side of the equality above, we arrive at the identty
$$
\ba
\om_y\d_y|\hat g(k,y,\om)|^2+2\si|\hat g(k,y,\om)-\la\hat g\ra(k,y)|^2
\\
=-2\si\Re\left(\overline{\hat g(k,y,\om)}\cL\hat g(k,y,\om)\right)&\,.
\ea
$$
Now we average both sides of this equality in $\om$: for each $k\in\bZ^d$, 
$$
\d_y\la\om_y|\hat g|^2\ra(k,y)+2\si\Re\la\overline{\hat g(k,y,\cdot)}\cL\hat g(k,y,\cdot)\ra=0\,,\quad y>0\,.
$$

At this point, we recall the following property of $\cL$. 

\begin{Lem}\lb{L-CoercL}
Under assumptions \eqref{H1p}-\eqref{H2p} on the function $p$, the operator $\cL$ satisfies the following properties. For each $\phi\in L^2(\bS^d;\bC)$, one has
$$
\la\overline\phi\cL\phi\ra=\tfrac1{2|\bS^d|}\iint_{\bS^d\times\bS^d}p(\om,\om')|\phi(\om)-\phi(\om')|^2\dd\om\dd\om'\,.
$$
Moreover, there exists $\mu>0$ such that
$$
\la\overline\phi\cL\phi\ra\ge\mu\|\phi-\la\phi\ra\|_{L^2(\bS^d)}^2\,.
$$
\end{Lem}

\smallskip
The proof of this lemma is deferred until the end of the present step.

\smallskip
Thus, for each $k\in\bZ^d$, the function
$$
y\mapsto\la\om_y|\hat g|^2\ra(k,y)
$$
is bounded and nonincreasing on $(0,+\infty)$, and applying Lemma \ref{L-CoercL} shows that
$$
2\si\mu\int_0^\infty\La\left|\hat g-\la\hat g\ra\right|^2\Ra(k,y)\dd y\le\lim_{Y\to+\infty}\left[\la\om_y|\hat g|^2\ra(k,\cdot)\right]_{y=Y}^{y=0}<+\infty\,.
$$

Next observe that
\be\lb{DecompEntroFlux}
\la\om_y|\hat g|^2\ra(k,y)=\la\om_y|\hat g-\la\hat g\ra|^2\ra(k,y)+2\Re\left(\la\om_y(\hat g-\la\hat g\ra)\ra(k,y)\overline{\la\hat g\ra(k,y)}\right)\,.
\ee
In view of the previous estimate, the first term on the right hand side of this equality belongs to $L^1(0,+\infty)$ since $|\om_y|\le 1$, while the second term is bounded by the Cauchy-Schwarz inequality:
$$
2\Re\left(\la\om_y(\hat g-\la\hat g\ra)\ra(k,y)\overline{\la\hat g\ra(k,y)}\right)\le\|g\|_{L^\infty}\la\om_y^2\ra^{1/2}\la|g-\la g\ra|^2\ra(k,y)^{1/2}\,.
$$
Hence the second term on the right hand side of \eqref{DecompEntroFlux} belongs to $L^2(0,+\infty)$, and since the first term is also bounded in $L^\infty(0,+\infty)$ because $g$ is assumed to belong to $L^\infty(Z\times\bS^d)$,
we conclude that
$$
y\mapsto\la\om_y|\hat g|^2\ra(k,y)\hbox{  belongs to }L^2(0,+\infty)
$$ 
for each $k\in\bZ^d$. Since this function is moreover nonincreasing, we conclude that
$$
0=\lim_{Y\to+\infty}\la\om_y|\hat g|^2\ra(k,Y)\le\la\om_y|\hat g|^2\ra(k,y)\quad\hbox{ for all }y>0\hbox{  and all }k\in\bZ^d\,.
$$

Thus, for each $Y>0$, one has
$$
\left\{
\ba
{}&\la\om_y|\hat g|^2\ra(k,Y)+2\si\mu\int_0^Y\La\left|\hat g-\la\hat g\ra\right|^2\Ra(k,y)\dd y\le\la\om_y|\hat g|^2\ra(k,0)\,,
\\
&\la\om_y|\hat g|^2\ra(k,Y)\ge 0\,.
\ea
\right.
$$
Moreover, the right hand side of this inequality satisfies
$$
\ba
\la\om_y|\hat g|^2\ra(k,0)=\tfrac{|\bB^d|}{|\bS^d|}\left(\lA|\hat g|^2\rA_+(k,0)-\lA|\hat g|^2\rA_-(k,0)\right)
\\
=\tfrac{|\bB^d|}{|\bS^d|}\left(\left(\tfrac{\b}{1+\b}\right)^2|\lA g\rA_-(k,0)|^2\lA1\rA_+-\lA|\hat g|^2\rA_-(k,0)\right)
\\
=\tfrac{|\bB^d|}{|\bS^d|}\left(\left(\tfrac{\b}{1+\b}\right)^2|\lA g\rA_-(k,0)|^2-\lA|\hat g|^2\rA_-(k,0)\right)
\\
\le\tfrac{|\bB^d|}{|\bS^d|}\left(\left(\tfrac{\b}{1+\b}\right)^2-1\right)\lA|\hat g|^2\rA_-(k,0)
\ea
$$
by the Cauchy-Schwarz inequality. 

Summarizing, we have proved that, for all $k\in\bZ^d$ and all $Y>0$
$$
0=\la\om_y|\hat g|^2\ra(k,Y)=\int_0^Y\La\left|\hat g-\la\hat g\ra\right|^2\Ra(k,y)\dd y=\lA|\hat g|^2\rA_-(k,0)\,.
$$
In particular, since $\cL 1=0$, one has
$$
\cL\hat g=\cL(\hat g-\la\hat g\ra)=0\,,
$$
so that
$$
\left\{
\ba
{}&\om_y\d_y\hat g(k,y,\om)+i2\pi k\cdot\om_x\hat g(k,y,\om)=g(k,y,\om)-\la\hat g\ra(k,y)=0&\,,
\\
&\hat g(k,0,\om)=\frac{\b}{1+\b}\lA g\rA_-(k,0)=0\,,\quad\,\,\om_y>0\,.
\ea
\right.
$$

The second equality above implies that, for a.e. $(k,y)\in\bZ^d\times(0,+\infty)$, the function $\om\mapsto\hat g(k,y,\om)$ is a.e. constant on $\bS^d$.  Hence
$$
\om_y\d_y\hat g(k,y)+i2\pi k\cdot\om_x\hat g(k,y)=0\,,\hbox{ for all }\om\in\bS^d\,,\hbox{ and a.e. }(k,y)\in\bZ\times(0,+\infty)\,.
$$
Setting $\om_x=k/|k|$ and $\om_y=0$ shows that
$$
k\not=0\Rightarrow\hat g(k,y)=0\hbox{ for a.e. }y\ge 0\,.
$$
On the other hand
$$
\d_y\hat g(0,y)=0
$$
by choosing $\om_x=0$ and $\om_y=1$, so that
$$
\hat g(0,y)=\hat g(0,0)\,.
$$
Finally, the boundary condition implies that
$$
\hat g(0,y)=\hat g(0,0)=0\,,\qquad y>0\,.
$$

Therefore, we have proved that $\hat g(k,y,\om)=0$ for a.e. $(y,\om)\in(0,+\infty)\times\bS^d$, for all $k\in\bZ^d$. Since the Fourier transform is one-to-one, this implies that $g=0$ in $L^\infty(Z\times\bS^d)$.

\smallskip
The discussion above shows the existence and uniqueness of the solution $f_\si\in L^\infty(Z\times\bS^d)$ to the boundary value problem \eqref{RTPb} in the case where $\b=\ka\si$; besides we have seen at the end of Step 3 that
\be\lb{Boundfsi}
\|f_\si\|_{L^\infty}\le\|S\|_{L^\infty}\,.
\ee
If moreover $S\in W^{1,\infty}(\bT^d)$, we may apply the results of Steps 1-4 to $\d_{x_j}S$; this shows that
$$
f_\si\in L^\infty((0,+\infty)\times\bS^d;W^{1,\infty}(\bT^d))\,,
$$
and satisfies the bound
\be\lb{BoundGradfsi}
\|\grad_xf_\si\|_{L^\infty}\le\|\grad_xS\|_{L^\infty}\,,\qquad\si>0\,.
\ee
Notice that this bound is uniform in $\si>0$. In other words, statement (a) in Theorem \ref{T-Ppal} is implied by steps 1-4 with $\b=\ka\si$.

\smallskip
\begin{proof}[Proof of Lemma \ref{L-CoercL}]
The formula for $\la\bar\phi\cL\phi\ra$ is \eqref{phiLpsi} where $\phi$ is replaced with $\bar\phi$ and $\psi$ with $\phi$. 

On the other hand, we already know from Lemma \ref{L-ImL} that $\cL$ is self-adjoint and of Fredholm type on $L^2(\bS^d;\bC^d)$. The Fredholm alternative implies that the continuous, one-to-one linear map
$$
\cL\rstr_{(\Ker\cL)^\perp}:\,(\Ker\cL)^\perp\to(\Ker\cL)^\perp
$$
is onto. By Banach's open mapping theorem (see Theorem 2.6 and Corollary 2.7 in \cite{Brezis}), this map is bicontinuous, which implies the existence of the positive constant $\mu$.
\end{proof}


\subsection{Step 5}


Set
$$
F_\si(x):=\frac{S(x)}{1+\si\nu}+\frac{\ka\si}{1+\ka\si}\lA f_\si\rA_-(x,0)\,;
$$
arguing as at the end of step 2, we find that
$$
\|F_\si\|_{L^\infty}\le\frac{\|S\|_{L^\infty}}{1+\ka\si}+\frac{\ka\si}{1+\ka\si}\lim_{\l\to 0^+}\|\cA_{\l,\si,\ka\si}F_\si\|_{L^\infty}\,,
$$
and then
$$
\|\cA_{\l,\si,\ka\si}F_\si\|_{L^\infty}\le\|\cA_{\l,\si,\ka\si}F_\si-\cA_{\l,\si,\ka\si}0\|_{L^\infty}+\|\cA_{\l,\si,\ka\si}0\|_{L^\infty}\le\frac{\ka\si}{1+\ka\si}\|F_\si\|_{L^\infty}\,,
$$
while
$$
\|\cA_{\l,\si,\ka\si}0\|_{L^\infty}\le\frac{\|S\|_{L^\infty}}{1+\ka\si}\,.
$$
Therefore
$$
\|F_\si\|_{L^\infty}\le\frac{\|S\|_{L^\infty}}{1+\ka\si}+\frac{\ka\si}{1+\ka\si}\left(\frac{\ka\si}{1+\ka\si}\|F_\si\|_{L^\infty}+\frac{\|S\|_{L^\infty}}{1+\ka\si}\right)
$$
and hence
$$
\left(1-\left(\frac{\ka\si}{1+\ka\si}\right)^2\right)\|F_\si\|_{L^\infty}\le\left(1+\frac{\ka\si}{1+\ka\si}\right)\frac{\|S\|_{L^\infty}}{1+\ka\si}\,.
$$
Since
$$
\left(1-\left(\frac{\ka\si}{1+\ka\si}\right)^2\right)=\left(1+\frac{\ka\si}{1+\ka\si}\right)\left(1-\frac{\ka\si}{1+\ka\si}\right)=\left(1+\frac{\ka\si}{1+\ka\si}\right)\frac1{1+\ka\si}
$$
the inequality above implies that
$$
\|F_\si\|_{L^\infty}\le\|S\|_{L^\infty}\,.
$$
By the same token
$$
\|\d_{x_j}F_\si\|_{L^\infty}\le\|\d_{x_j}S\|_{L^\infty}\,,\qquad j=1,\ldots,d\,.
$$

Let $\chi(y):=(1-y)_+^2$, so that $\chi\in C^1([0,\infty))$ has its support equal to $[0,1]$, and set
$$
g_\si(x,y):=F_\si(x)\chi(y)\,;
$$
define further the function
$$
h_\si(x,y,\om):=f_\si(x,y,\om)-g_\si(x,y)\,.
$$
It is easily seen that $h_\si$ is a solution to the boundary value problem
\be\lb{BVPh}
\left\{
\ba
{}&\om\cdot\grad_zh_\si+\si\cL h_\si=-\om\cdot\grad g_\si(z)\,,
\\
&h_\si(x,0,\om)=0\,,\quad\om_y>0\,.
\ea
\right.
\ee
Notice that $h_\si\in L^\infty((0,+\infty)\times\bS^d;W^{1,\infty}(\bT^d))$. 

Because of the regularity of $h_\si$ in $x$, we deduce from the equation that the function $z\mapsto h_\si(z,\om)$ belongs to $W^{1,\infty}(\bT^d\times(0,+\infty))$. Thus, we can multiply both sides of the first equation in \eqref{BVPh} 
by $2\si h_\si$ and integrate in $\om$, to obtain
$$
\si\Div_z\la\om h^2_\si\ra+2\si^2\la h_\si\cL h_\si\ra=-2\si\la\om h_\si\ra\cdot\grad g_\si(z)=-2\si\la\om(h_\si-\la h_\si\ra)\ra\cdot\grad g_\si(z)\,.
$$
After integrating in $x\in\bT^d$ both sides of this equality, we see that
$$
\ba
\si\d_y\int_{\bT^d}\la\om_yh^2_\si\ra(x,y)\dd x+2\si^2\int_{\bT^d}\la h_\si\cL h_\si\ra(x,y)\dd x
\\
=-2\si\int_{\bT^d}\la\om(h_\si-\la h_\si\ra)\ra\cdot\grad g_\si(x,y)\dd x\,.
\ea
$$

By construction
$$
f_\si(x,y,\om)=h_\si(x,y,\om)\quad\hbox{ whenever }y>1\,,
$$
so that
$$
\la\om_yh^2_\si\ra(x,y)=\la\om_yf^2_\si\ra(x,y)\hbox{ and }g_\si(x,y)=0\,,\quad\hbox{ for all }y>1\,.
$$
On the other hand, Parseval's theorem implies that
$$
\int_{\bT^d}\la\om_yf^2_\si\ra(x,y)\dd x=\La\om_y\int_{\bT^d}f^2_\si(x,y,\cdot)\dd x\Ra=\sum_{k\in\bZ^d}\la\om_y|\hat f_\si|^2\ra(k,y)\,.
$$
Proceeding as in step 4 and applying Lemma \ref{L-CoercL}, we see that
$$
\d_y\la\om_y|\hat f_\si|^2\ra(k,y)+2\si\mu\La\left|\hat f_\si-\la\hat f_\si\ra\right|^2\Ra(k,y)\le 0\,,\quad y>1\,,\,\,k\in\bZ^d\,.
$$
For all $k\in\bZ^d$, the function
$$
y\mapsto\la\om_y|\hat f_\si|^2\ra(k,y)
$$
is bounded and Lipschitz continuous on $(0,+\infty)$, and nonincreasing on $(1,+\infty)$. Hence
$$
\int_1^\infty\La\left|\hat f_\si-\la\hat f_\si\ra\right|^2\Ra(k,y)\dd y\le\frac1{2\si\mu}\left[\la\om_y|\hat f_\si|^2\ra(k,y)\right]^{y\,=\,1}_{y\to\infty}<\infty\,.
$$
With the decomposition
$$
\la\om_y|\hat f_\si|^2\ra(k,y)=\la\om_y|\hat f_\si-\la\hat f_\si\ra|^2\ra(k,y)+2\Re\left(\la\om_y(\hat f_\si-\la\hat f_\si\ra)\ra(k,y)\overline{\la\hat f_\si\ra(k,y)}\right)\,,
$$
we conclude that
$$
y\mapsto\la\om_y|\hat f_\si|^2\ra(k,y)\hbox{ belongs to }L^2(0,+\infty)
$$ 
for all $k\in\bZ^d$. (Notice that we do not seek a uniform in $\si$ estimate in $L^2$ for $\hat f_\si-\la\hat f_\si\ra$ at this stage in the argument, although this is our ultimate goal.) Since this function is nonincreasing on $(1,+\infty)$ 
for all $k\in\bZ^d$, we find that
$$
0=\lim_{Y\to+\infty}\la\om_y|\hat f_\si|^2\ra(k,Y)\le\la\om_y|\hat f_\si|^2\ra(k,y)\quad\hbox{ for all }y>1\hbox{  and all }k\in\bZ^d\,.
$$

Thus, for each $Y>1$, one has
$$
\ba
\si\int_{\bT^d}\la\om_yf^2_\si\ra(x,Y)\dd x+2\si^2\mu\int_0^Y\int_{\bT^d}\La(h_\si-\la h_\si\ra)^2\Ra(x,y)\dd x\dd y
\\
=-2\si\int_0^1\int_{\bT^d}\la\om(h_\si-\la h_\si\ra)\ra\cdot\grad g_\si(x,y)\dd x\dd y+\si\int_{\bT^d}\la\om_yh^2_\si\ra(x,0)\dd x&\,,
\ea
$$
so that
\be\lb{h-avh}
\ba
2\si^2\mu\int_0^Y\int_{\bT^d}\La(h_\si-\la h_\si\ra)^2\Ra(x,y)\dd x\dd y
\\
\le-2\si\int_0^1\int_{\bT^d}\la\om(h_\si-\la h_\si\ra)\ra\cdot\grad g_\si(x,y)\dd x\dd y&\,.
\ea
\ee
Indeed
$$
\la\om_yh^2_\si\ra(x,0)=-\tfrac{|\bB^d|}{|\bS^d|}\lA h^2_\si\rA_-(x,0)\le 0\,,
$$
since $h_\si(x,0,\om)=0$ for $\om_y>0$ by construction, and on the other hand, as explained above,
$$
\int_{\bT^d}\la\om_yf^2_\si\ra(x,Y)\dd x\ge 0\quad\hbox{ for all }Y>1\,.
$$
Notice that the first integral on the right hand side of \eqref{h-avh} involves only $y\in[0,1]$ since $g(x,y)=0$ for all $y>1$ by construction.

By the Cauchy-Schwarz inequality, for all $Y>1$, one has
$$
\ba
\si^2\mu\int_0^Y\int_{\bT^d}\La(h_\si-\la h_\si\ra)^2\Ra(x,y)\dd x\dd y
\\
\le\!\!\left(\si^2\int_0^1\int_{\bT^d}\left|\la\om(h_\si-\la h_\si\ra)\ra\right|^2(x,y)\dd x\dd y\right)^{1/2}\!\!\left(\int_0^1\int_{\bT^d}|\grad g_\si(x,y)|^2\dd x\dd y\right)^{1/2}
\\
\le\!\!\left(\si^2\int_0^Y\int_{\bT^d}\La(h_\si-\la h_\si\ra)^2\Ra(x,y)\dd x\dd y\right)^{1/2}\!\!\left(\int_0^1\int_{\bT^d}|\grad g_\si(x,y)|^2\dd x\dd y\right)^{1/2}&,
\ea
$$
so that
$$
\si^2\int_0^Y\int_{\bT^d}\La(h_\si-\la h_\si\ra)^2\Ra(x,y)\dd x\dd y\le\int_0^1\int_{\bT^d}|\grad g_\si(x,y)|^2\dd x\dd y\,.
$$
Observe that
$$
h_\si-\la h_\si\ra=f_\si-\la f_\si\ra
$$
since $g_\si=f_\si-h_\si$ is independent of $\om$ by construction. Hence
$$
\si^2\mu\int_0^Y\int_{\bT^d}\La(h_\si-\la h_\si\ra)^2\Ra(x,y)\dd x\dd y\le\int_0^1\int_{\bT^d}|\grad g_\si(x,y)|^2\dd x\dd y\,,
$$
and since this inequality holds for all $Y>0$, we conclude that
$$
\si^2\mu\int_0^\infty\int_{\bT^d}\La(h_\si-\la h_\si\ra)^2\Ra(x,y)\dd x\dd y\le\int_0^1\int_{\bT^d}|\grad g_\si(x,y)|^2\dd x\dd y\,.
$$


\subsection{Step 6}


Summarizing, for each $S\in W^{1,\infty}(\bT^d)$, the boundary value problem \eqref{RTPb} has a unique (by Step 4) solution $f_\si\in L^\infty(Z\times\bS^d)$, and this solution satisfies
$$
\|f_\si\|_{L^\infty}\le\|S\|_{L^\infty}\,,
$$
according to Step 3, and
$$
\|\si(f_\si-\la f_\si\ra)\|^2_{L^2(Z\times\bS^d)}\le\tfrac{2}{\mu}|\bS^d|\left(\|\grad S\|_{L^\infty(\bT^d)}^2+\tfrac43\|S\|^2_{L^\infty(\bT^d)}\right)\,,
$$
by Step 5.

\smallskip
Therefore, there exists a sequence $\si_n\to+\infty$ such that
$$
f_{\si_n}\to f\quad\hbox{ in }L^\infty(Z\times\bS^d)\hbox{ weak}-*\,,
$$
and
$$
J_n:=\si_n(f_{\si_n}-\la f_{\si_n}\ra)\to J\quad\hbox{ in }L^2(Z\times\bS^d)\hbox{ weak}\,.
$$
In particular, the fact that $J_n$ is bounded in $L^2(Z\times\bS^d)$ implies that
\be\lb{Flarho}
f(z,\om)=\rho(z)\hbox{ for a.e. }(z,\om)\in Z\times\bS^d\,.
\ee

\smallskip
In order to compute $J$, observe that, since $\cL 1=0$,
$$
\cL J_n=\si_n\cL f_{\si_n}=-\om\cdot\grad_zf_{\si_n}\,.
$$
Since the linear operator $\cL$ is continuous on $L^2(\bS^d)$, one has
$$
\cL J_n\to\cL J\hbox{  in }L^2(Z\times\bS^d)\hbox{ weak,}
$$
and since
$$
\om\cdot\grad_zf_{\si_n}\to\om\cdot\grad\rho\quad\hbox{ in }\cD'(Z\times\bS^d)\,,
$$
we conclude that
$$
\cL J(z,\om)=\om\cdot\grad\rho(z)\hbox{ for a.e. }(z,\om)\in Z\times\bS^d\,.
$$
Beside, since $J_n(z,\cdot)\in(\Ker\cL)^\perp$ for all $n$ and a.e. $z\in Z$, one has also
$$
J(z,\cdot)\in(\Ker\cL)^\perp\,,\qquad\hbox{ for a.e. }z\in Z\,.
$$
Applying the Fredholm alternative to $\cL$, we conclude that
\be\lb{FlaJ}
J(z,\om)=-\Om(\om)\cdot\grad\rho(z)\,,\qquad\hbox{ for a.e. }(z,\om)\in Z\times\bS^d\,.
\ee
Notice that the formula above for $J$ and the $L^2$ bound on $J_n$ imply that
$$
\grad\rho\in L^2(Z)\,.
$$

Let now $\phi\in C^\infty_c(\bT^d\times(0,+\infty))$. Multiplying both sides of the equation for $f_{\si_n}$ in \eqref{RTPb} by
$$
\si_n\phi(z)+\Om(\om)\cdot\grad\phi(z)
$$
and integrating both sides in $(z,\om)$, we obtain
\be\lb{IdentLim}
\ba
\int_0^\infty\int_{\bT^d}\left(\phi(z)\Div_z(\si_n\la\om f_{\si_n}\ra)+\grad\phi(z)\cdot\la\Om(\om\cdot\grad_zf_{\si_n})\ra\right)\dd x\dd y
\\
+\int_0^\infty\int_{\bT^d}\si_n^2\phi(z)\la\cL f_{\si_n}\ra\dd x\dd y+\int_0^\infty\int_{\bT^d}\grad\phi(z)\cdot\la\si_n\Om\cL f_{\si_n}\ra\dd x\dd y=0&\,.
\ea
\ee
Since $\cL$ is self-adjoint on $L^2(\bS^d)$ (see Lemma \ref{L-ImL}), one has
$$
\la\cL f_{\si_n}\ra=\la f_{\si_n}\cL^*1\ra=0\,,
$$
and
$$
\la\si_n\Om\cL f_{\si_n}\ra=\la\si_nf_{\si_n}\cL\Om\ra=\la\si_n\om f_{\si_n}\ra\,.
$$
Hence the third term on the left hand side of \eqref{IdentLim} is $0$, while the first and last terms on the left hand side of \eqref{IdentLim} combine to give
$$
\int_0^\infty\int_{\bT^d}\left(\Div_z(\si_n\la\om f_{\si_n}\ra\phi(z))+\grad\phi(z)\cdot\la\Om(\om\cdot\grad_zf_{\si_n})\ra\right)\dd x\dd y=0\,.
$$

The first integral on the left hand side is simplified by using Green's formula:
$$
\int_0^\infty\int_{\bT^d}\Div_z(\si_n\la\om f_{\si_n}\ra\phi(z))\dd x\dd y=-\int_{\bT^d}\si_n\la\om f_{\si_n}\ra\phi(x,0)\dd x\,.
$$
Next we use the boundary condition verified by $f_{\si_n}$, i.e.
$$
f_{\si_n}(x,0,\om)=\frac{S(x)}{1+\ka\si_n}+\frac{\ka\si_n}{1+\ka\si_n}\lA f_{\si_n}\rA_-(x)\,,\quad\om_y>0\,,
$$
or equivalently
$$
f_{\si_n}(x,0,\om)=S(x)-\ka\si_n\tfrac{|\bS^d|}{|\bB^d|}\la\om f_{\si_n}\ra(x,0)\,,\quad\om_y>0\,.
$$
Since the right hand side of this identity is independent of $\om$, one has also
\be\lb{BCfsin}
\lA f_{\si_n}\rA_+(x,0)=S(x)-\ka\si_n\tfrac{|\bS^d|}{|\bB^d|}\la\om f_{\si_n}\ra(x,0)\,,\quad x\in\bT^d\,.
\ee
Thus
$$
\ba
\int_0^\infty\int_{\bT^d}\Div_z(\si_n\la\om f_{\si_n}\ra\phi(z))\dd x\dd y=-\int_{\bT^d}\si_n\la\om f_{\si_n}\ra\phi(x,0)\dd x
\\
=\tfrac{|\bB^d|}{\ka|\bS^d|}\int_{\bT^d}(S(x)-\lA f_{\si_n}\rA_+(x,0))\phi(x,0)\dd x&\,.
\ea
$$
In other words
$$
\ba
\tfrac{|\bB^d|}{\ka|\bS^d|}\int_{\bT^d}(S(x)-\lA f_{\si_n}\rA_+(x,0))\phi(x,0)\dd x
\\
=\int_0^\infty\int_{\bT^d}\grad\phi(z)\cdot\la\Om(\om\cdot\grad_zf_{\si_n})\ra\dd x\dd y
\ea
$$
for all $\phi\in C^\infty_c(\bT^d\times(0,+\infty))$.

As $n\to\infty$, the integral on the right hand side satisfies
$$
\ba
\int_0^\infty\int_{\bT^d}\grad\phi(z)\cdot\la\Om(\om\cdot\grad_zf_{\si_n})\ra\dd x\dd y=-\int_0^\infty\int_{\bT^d}\grad\phi(z)\cdot\la\Om\cL J_n\ra\dd x\dd y
\\
\to-\int_0^\infty\int_{\bT^d}\grad\phi(z)\cdot\la\Om\cL J\ra\dd x\dd y=-\int_0^\infty\int_{\bT^d}\grad\phi(z)\cdot\la(\cL^*\Om)J\ra\dd x\dd y
\\
=-\int_0^\infty\int_{\bT^d}\grad\phi(z)\cdot\la\om J\ra\dd x\dd y=\int_0^\infty\int_{\bT^d}\la\om\otimes\Om\ra:\grad\phi(z)\otimes\grad_z\rho(z)\ra\dd x\dd y
\\
=\frac{\la\om\cdot\Om\ra}{d+1}\int_0^\infty\int_{\bT^d}\grad\phi(z)\cdot\grad_z\rho(z)\ra\dd x\dd y&\,.
\ea
$$

On the other hand, arguing as in Step 3 shows that
$$
\ba
\tfrac{|\bB^d|}{|\bS^d|}\d_y\int_{\bT^d}\lA f_{\si_n}\rA_+(x,y)\psi(x)\dd x=\d_y\int_{\bT^d}\la\om^+_yf_{\si_n}\ra(x,y)\psi(x)\dd x
\\
=\int_{\bT^d}\la\indc_{\om_y>0}\om_xf_{\si_n}\ra(x,y)\psi'(x)\dd x-\int_{\bT^d}\la\indc_{\om_y>0}\cL J_n\ra(x,y)\psi(x)\dd x
\\
=O(1)\hbox{ in }L^2+L^\infty(0,+\infty)&\,.
\ea
$$
By Ascoli-Arzel\`a theorem, we conclude that
$$
\int_{\bT^d}\lA f_{\si_n}\rA_+(x,0)\phi(x,0)\dd x\to\int_{\bT^d}\rho(x,0)\phi(x,0)\dd x\,.
$$

Summarizing, we have proved that
$$
\ba
\rho\in L^\infty(Z)\,,\quad\grad\rho\in L^2(Z)
\\
\tfrac{|\bB^d|}{\ka|\bS^d|}\int_{\bT^d}(S(x)-\rho(x,0))\phi(x,0)\dd x=\frac{\la\om\cdot\Om\ra}{d+1}\int_0^\infty\int_{\bT^d}\grad\phi(z)\cdot\grad_z\rho(z)\ra\dd x\dd y
\ea
$$
for all $\phi\in C^\infty_c(\bT^d\times(0,+\infty))$. One easily checks that this is the variational formulation of the boundary value problem \eqref{Diff}.

\smallskip
This proves statement (b) of Theorem \ref{T-Ppal}. 


\subsection{Step 7}


Set $Y=\bT^d\times(0,1)$. The results obtained in Step 6 imply that
$$
f_{\si_n}\rstr_{Y\times\bS^d}\to\rho\quad\hbox{  and }\quad\om\cdot\grad_zf_{\si_n}\rstr_{Y\times\bS^d}\to\om\cdot\grad_z\rho
$$
weakly in $L^2(Y\times\bS^d)$.

By applying Cessenat's trace theorem (Theorem 1 in \cite{Cessenat1}), we find that
\be\lb{ConvTrace}
f_{\si_n}(x,0,\om)\to R(x)\hbox{ weakly in }L^2(\bT^d\times\bS^d;\om_y^+\dd\om\dd x)\,,
\ee
where
$$
R=\rho\rstr_{y=0}\,.
$$

This last point requires additional explanations.  At this point, we use Cessenat's notation in \cite{Cessenat1}. Observe that, for all $x\in\bT^d$, the exit time from $Y$, starting from $y$ in the direction $\om$, is 
$$
\tau_{(x,0),\om}=1/|\om_y|\ge 1\,,\quad\hbox{ since }|\om|=1\,.
$$
Choosing the arbitrary parameter $K=1$ in the definition of the measure $\dd\xi$ on p. 832 in \cite{Cessenat1}, viz.
$$
\dd\xi\rstr_{\bT^d\times\{0\}\times\bS^d}=|\om_y|\min(1,\tau_{(x,0),\om})\dd x\dd\om\,,
$$
has its restriction to $\bT^d\times\{0\}$ given by
$$
\dd\xi\rstr_{\bT^d\times\{0\}\times\bS^d}=|\om_y|\dd x\dd\om\,.
$$
This observation justifies that the convergence holds as stated in \eqref{ConvTrace}. 

Since
$$
\rho\in L^\infty(0,+\infty;L^2(\bT^d))\quad\hbox{ and }\grad\rho\in L^2(Z)\,,
$$
one has in particular $\rho\rstr_Y\in H^1(Y)$, so that
$$
R=\rho\rstr_{y=0}\in H^{1/2}(\bT^d)\,.
$$
Since
$$
\grad\rho\in L^2(Z)\hbox{ and }\Div(\grad\rho)=0\,,
$$
we deduce from the trace theorem of J.-L. Lions for the space $H(Y,\Div)$ of vector fields in $L^2(Y;\bR^{d+1})$ with divergence in $L^2(Y)$ (see Lemma 20.2 in \cite{TartarInterp}) that
$$
\d_y\rho\rstr_{y=0}\in H^{-1/2}(\bT^d)\,.
$$

On the other hand
$$
\si_n\la\om f_{\si_n}\ra\rstr_{Y}=\si\la\om J_n\ra\rstr_{Y}\hbox{ is bounded in }L^\infty(0,1;L^2(\bT^d))\subset L^2(Y)\,,
$$
and
$$
\Div_z(\si_n\la\om f_{\si_n}\ra)=0\,.
$$
Applying again the Lions trace theorem in $H(Y,\Div)$, we conclude that
$$
\si_n\la\om_y f_{\si_n}\ra\rstr_{y=0}\to-\frac{\la\om\cdot\Om\ra}{d+1}\d_y\rho\rstr_{y=0}\hbox{ in }H^{-1/2}(\bT^d)\hbox{ weak.}
$$

Since $\grad\rho\in L^2(Z)$, we find that
$$
\d_y\rho\in L^2(\bT^d\times(0,\infty))\,,
$$
so that
$$
\ba
\int_{\bT^d}|\rho(x,y_2)-\rho(x,y_1)|^2\dd x=\int_{\bT^d}\left(\int_{y_1}^{y_2}\d_y\rho(x,z)\dd z\right)^2\dd x
\\
\le|y_2-y_1|\int_{\bT^d}\int_{y_1}^{y_2}|\d_y\rho(x,z)|^2\dd z\dd x\le|y_2-y_1|\|\grad\rho\|_{L^2(Z)}&\,.
\ea
$$
Hence
$$
\rho\in C_b(0,+\infty;L^2(\bT^d))\,.
$$
Since $\rho$ is a harmonic extension of $R=\rho\rstr_{y=0}$, we deduce from Kwasnicki's Theorem 1.1 (j) that
$$
\d_y\rho\rstr_{y=0}=-(-\Dlt_x)^{1/2}R\,.
$$
This concludes the proof of statement (c) of Theorem \ref{T-Ppal}.


\section{Open Problems}


The result obtained in the present paper suggests various questions, still open at the time of this writing.

\bigskip
\noindent
(a) Can one extend Theorem \ref{T-Ppal} to other kinetic models --- for instance, to the linearized Boltzmann equation, or even to the linearized BGK model? For instance, one could consider the Boltzmann equation for a vapor in a half-space
over its liquid, condensed phase, with a linear combination of diffuse reflection and condensation or evaporation at the boundary, assuming that the Knudsen number in the vapor is small, and that the accomodation coefficient is close to one.
At present, the proof of Theorem \ref{T-Ppal} is based on the maximum principle, which the linearized Boltzmann equation does not satisfy. Before treating the case of the linearized Boltzmann equation, it would be necessary to have a proof
of Theorem \ref{T-Ppal} based on an $L^2$ energy estimate.

\smallskip
\noindent
(b) Can one obtain in this way powers of the Laplacian other than $(-\Dlt)^{1/2}$? The harmonic extension result recalled in section \ref{SS-HarmExt} suggests that one should seek a linearized Boltzmann equation leading to the diffusion 
operator
$$
y^{\frac2\g-2}\Dlt_x+\g^2c_\g^{\g/2}\d^2_y\,.
$$
The scattering operator of such a linearized Boltzmann equation must be strongly anisotropic; besides, the scattering coefficient should vanish either as $y\to+\infty$ or as $y\to 0$ depending on whether $\g\in(0,1)$ or $\g\in(1,2)$.

\smallskip
\noindent
(c) Can one extend the result in Theorem \ref{T-Ppal} to other domains than half-spaces? Assuming that one can derive from the linearized Boltzmann equation in a bounded (spatial) domain $\Om$ with smooth boundary the diffusion
problem
$$
\left\{
\ba
{}&\Dlt\rho(x)=0\,,&\quad x\in\Om\,,
\\
&\left(\rho+\ka\frac{\d\rho}{\d n}\right)\Bigg|_{\d\Om}=S\,,&
\ea
\right.
$$
the trace $\rho\rstr_{\d\Om}$ must satisfy the equation
$$
\rho\rstr_{\d\Om}+\ka\L\rho\rstr_{\d\Om}=S\,,
$$
where $\L$ is the Dirichlet-to-Neumann operator, defined as follows. 

For $u\in H^{1/2}(\d\Om)$, we set
$$
\L u:=\frac{\d U}{\d n}\Bigg|_{\d\Om}
$$
where $U\in H^1(\Om)$ is the unique solution of
$$
-\Dlt U=0\hbox{ in }\Om\,,\qquad U\rstr_{\d\Om}=u\,.
$$
We recall that the linear map $u\mapsto U$ is continuous from $H^{1/2}(\d\Om)$ to $H^1(\Om)$. By Green's formula
$$
\La \frac{\d U}{\d n},\phi\Ra:=\int_{\Om}\grad U(x)\cdot\grad\phi(x)\dd x\,,
$$
so that, for all $\phi\in H^1(\Om)$
$$
\left|\La \frac{\d U}{\d n},\phi\Ra\right|\le\|\grad U\|_{L^2(\Om)}\|\grad\phi\|_{L^2(\Om)}\le C\|u\|_{H^{1/2}(\d\Om)}\grad\phi\|_{L^2(\Om)}\,.
$$
This defines $\d U/\d n\rstr_{\d\Om}$ as an element of $H^{-1/2}(\d\Om)$, and shows that the linear map $\L$ is continuous from $H^{1/2}(\d\Om)$ to $H^{-1/2}(\d\Om)$. Observe that $\L$ is self-adjoint since
$$
\la\cL u,\phi\ra=\int_{\Om}\grad u(x)\cdot\grad\phi(x)\dd x=\la\cL \phi,u\ra\,.
$$
Besides, $\L\ge 0$ and $\Ker\cL=\bR=\{\hbox{ constants }\}$, since
$$
\la\cL u,u\ra=\|\grad u\|_{L^2(\Om)}^2\,.
$$

Call $\Dlt^{\d\Om}$ the Laplace-Beltrami operator\footnote{At variance with the definition commonly used in geometry, we normalize this operator so that $\Dlt^{\d\Om}\le 0$.} for the restriction of the Euclidean metric to $\d\Om$. Both $\L$ and 
$(-\Dlt^{\d\Om})^{1/2}$ are unbounded self-adjoint operators on $L^2(\d\Om)$, mapping $H^{1/2}(\d\Om)$ to $H^{-1/2}(\d\Om)$. 

However, $\L$ and $(-\Dlt^{\d\Om})^{1/2}$ are different in general, as shown by the following example. 

\smallskip
\noindent
\textbf{Example.} Let $\Om$ be the open unit ball of $\bR^{d+1}$, so that $\d\Om=\bS^d$ is the $d$-dimensional unit sphere. Using spherical coordinates, i.e. setting $r=|z|$ and $\om:=z/|z|$ for $z\not=0$, one has
$$
\Dlt_zf(z)=r^{-d}\d_r(r^d\d_rf(r\om))+r^{-2}\Dlt^{\bS^d}_\om f(r\om)\,.
$$
Denote by $0=\l_0<\l_1\le\l_2\ldots$ the sequence of eigenvalues of $\Dlt^{\bS^d}$ counted with multiplicities, and let $(e_n)_{n\ge 0}$ be a orthonormal and complete system in $L^2(\d\Om)$, with $-\Dlt^{\bS^d}e_n=\l_ne_n$. For each
$f\equiv f(r,\om)\in L^2(\d\Om)$, set
$$
f_n(r):=(e_n|f(r,\cdot))_{L^2(\d\Om)}\,,\qquad n\in\bN\,,\,\,r>0\,.
$$
Then $\Dlt_zf=0$ on $\Om$ if and only if
$$
r^{2-d}\d_r(r^d\d_rf_n(r))-\l_nf_n(r)=0\,,\quad n\in\bN\,,\,\,r>0\,,
$$
or equivalently
$$
r^2f_n''(r)+drf'_n(r)-\l_nf_n(r)=0\,,\quad n\in\bN\,,\,\,r>0\,.
$$
This differential equation has a two-dimensional space of solutions over the half-line $(0,+\infty)$, viz.
$$
\{C_+r^{\a_+(n)}+C_-r^{\a_-(n)}\,,\,C_\pm\in\bR\}
$$
where
$$
\a_\pm(n):=-\tfrac{d-1}2\pm\sqrt{\tfrac{(d-1)^2}{4}+\l_n}\,,\quad n\ge 0\,.
$$
Observe that $\a_-\le 1-d$ for all $n\in\bN$, since $\l_n\ge 0$. Hence $2(\a_--1)+d\le-d\le 1$, so that
$$
\int_0^1|r^{\a_--1}|^2r^d\dd r=+\infty\,.
$$
Hence all harmonic functions in $H^1(\Om)$ in the unit ball $\Om$ are of the form
$$
f(r\om)=\sum_{n\in\bN}c_nr^{\a_+(n)}e_n(\om)\,,
$$
and
$$
\d_rf(r\om)\rstr_{r=1}=\sum_{n\in\bN}c_n\a_-(n)e_n(\om)\quad\hbox{ while }\quad f(\om)=\sum_{n\in\bN}c_ne_n(\om)\,.
$$
The formula for $\a_+(n)$ indicates that, in this case
$$
\L=-\tfrac{d-1}2+\sqrt{\tfrac{(d-1)^2}{4}-\Dlt^{\bS^d}}\,.
$$
In particular
$$
\L\not=(-\Dlt^{\bS^d})^{1/2}\,,\quad\hbox{ unless }d=1\,.
$$


\end{document}